\documentclass[a4paper, 10pt, reqno]{amsart}
\usepackage{amsmath}
\usepackage{amssymb}
\usepackage{amsthm}
\usepackage{mathtools, verbatim}
\usepackage[utf8]{inputenc}
\usepackage{indentfirst}
\usepackage{enumerate}
\usepackage[colorlinks=true,hidelinks]{hyperref}
\hypersetup{urlcolor=blue, citecolor=red}
\usepackage{graphicx}
\usepackage[usenames,dvipsnames]{pstricks}
\usepackage{epsfig}
\usepackage{pst-grad} 
\usepackage{pst-plot} 
\usepackage{ifthen}
\usepackage{color,xcolor}
\usepackage[numbers]{natbib}
\usepackage{bookmark}

\def\sideremark#1{\ifvmode\leavevmode\fi\vadjust{\vbox to0pt{\vss
 \hbox to 0pt{\hskip\hsize\hskip1em
 \vbox{\hsize2.1cm\tiny\raggedright\pretolerance10000
  \noindent #1\hfill}\hss}\vbox to15pt{\vfil}\vss}}}%

\allowdisplaybreaks

\usepackage{fullpage}
\numberwithin{equation}{section}

\def\polhk#1{\setbox0=\hbox{#1}{\ooalign{\hidewidth\lower1.5ex\hbox{`}\hidewidth\crcr\unhbox0}}}

\def\XXint#1#2#3{{\setbox0=\hbox{$#1{#2#3}{\int}$ }
\vcenter{\hbox{$#2#3$ }}\kern-.6\wd0}}

\renewcommand{\div}{\operatorname{div}}

\newcommand{\osc}{\operatorname{osc}}

\renewcommand{\div}{\operatorname{div}}

\newcommand{\tr}{\operatorname{Tr}}

\newcommand{\R}{\mathbb{R}}
\newcommand{\N}{\mathbb{N}}

\theoremstyle{plain}
\newtheorem{theorem}{Theorem}[section]

\newtheorem{definition}{Definition}[section]
\newtheorem{lemma}{Lemma}[section]

\newtheorem{proposition}{Proposition}[section]
\newtheorem{remark}{Remark}[section]
\newtheorem{Assumption}{A}

\newcommand{\NDelta}{\Delta_p^{\mathrm{N}}}

\title{Regularity for the normalized p-Laplacian equation with an arbitrary degeneracy law}

\author{Claudemir Alcantara and Makson Santos}
 
\begin{document}


  
\begin{abstract} 
We examine the interior regularity of solutions to a degenerate normalized $p$-Laplace equation, where the degeneracy is governed by a modulus of continuity whose inverse satisfies a Dini continuity condition. We prove that under very general assumptions on the degeneracy law, solutions belong to the $C^1$ class. We argue by approximating the solutions by a sequence of hyperplanes, which allows us to prove the desired regularity.
\end{abstract}   
 
\date{\today}

\maketitle
 
\tableofcontents


\section{Introduction}\label{sct intro}

In this paper, we investigate the regularity theory of viscosity solutions to a degenerate normalized $p$-Laplace equation of the form

\begin{equation}\label{1}
-\sigma(|Du(x)|)\NDelta u(x)=f(x) \quad \mbox{ in} \quad B_1,
\end{equation}
where $\sigma(\cdot)$ is a modulus of continuity, $p \in (1, +\infty)$ and $f \in {C(B_1)}\cap L^\infty(B_1)$. Under suitable conditions on $\sigma(\cdot)$, we prove that viscosity solutions to \eqref{1} are differentiable.   

The normalized $p$-Laplace operator is given by
\begin{align*}
\NDelta u & = |Du|^{2-p}\Delta_pu = \Delta u + (p-2)\Delta_\infty u \\
& = \Delta u + (p-2)\left\langle D^2u\dfrac{Du}{|Du|},\dfrac{Du}{|Du|} \right\rangle,
\end{align*}
where $\Delta_pu := \div(|Du|^{p-2}Du)$ is the usual $p$-Laplace operator. Consequently, the normalized $p$-Laplace operator can be seen as a uniformly elliptic operator in nondivergence form, with elliptic constants $\min\{p-1, 1\}$ and $\max\{p-1, 1\}$. It is important to note that, since this is a gradient-dependent model with a singularity at $\{Du = 0\}$, classical $C^{1, \alpha}$ results from the literature (e.g., \cite{caffarelli89, CC95}) do not apply directly.

The operator in \eqref{1} appears in many areas of mathematics, ranging from stochastic processes to differential geometry. For instance, \cite{Evans-Spruck} provides a geometric characterization related to the mean curvature flow, while \cite{Peres-Sheffield} offers a stochastic interpretation via tug-of-war games with noise and running pay-off. As a result, the normalized $p$-Laplace operator has been extensively studied over the years. Using a game-theoretic approach, the authors in \cite{LPS13} established new regularity estimates for $\NDelta u = 0$, which was later extended to the case of bounded, positive $f$ in \cite{Ruo16}. We refer the reader to \cite{CPCM13} for a PDE-based approach. We also mention that the parabolic version of the mentioned results can be found in \cite{BG15, Does2011, Jin-Silvestre2017, Manfredi-Parviainen-Rossi}.

Gradient estimates for the viscosity solutions to
\begin{equation}\label{eq_npl}
-\NDelta u = f \quad \mbox{ in} \quad B_1,    
\end{equation}
was studied in \cite{APR17}, where the authors showed that viscosity solutions to \eqref{eq_npl} are locally of class $C^{1, \alpha}$, for some $\alpha \in (0,1)$. They also derived estimates depending on the weaker $\|f\|_{L^q(B_1)}$ norm of $f$, provided $p>2$ and $q > \max\{2, d, p/2\}$. We also would like to mention the borderline regularity obtained in \cite{BM20}, where the author proved that $W^{2,q}$-viscosity solutions to \eqref{eq_npl}, with $f \in L(n, 1)$ (the standard Lorentz space), are differentiable. Moreover, if $f \in L^q$, and $q > d$, then solutions are $C_{loc}^{1, \alpha}(B_1)$, for some $\alpha \in (0,1)$.  

In the past decade, degenerate or singular normalized $p$-Laplace models of the form
\begin{equation}\label{eq_dnpl}
-|Du|^\gamma\NDelta u = f(x) \quad \mbox{ in} \quad B_1,    
\end{equation}
where $\gamma > -1$, have received significant attention. Notice that this formulation extends both the $p$-Laplace operator (for $\gamma = p-2$) and the normalized $p$-Laplace operator (for $\gamma = 0$). 

Regarding the regularity of viscosity solutions to \eqref{eq_dnpl}, the authors in \cite{birregrad} established that radial solutions belong to the class $C^{1, \alpha}$. Later, in \cite{Attouchi-Ruosteenoja2018}, this result was generalized to all viscosity solutions for $f \in C(B_1)\cap L^\infty(B_1)$, proving that they are locally of class $C^{1, \alpha}(B_1)$. Additionally, when $\gamma \in [-1,0)$ and $p \in \left(1, 3 + (2/n-2)\right)$, solutions were shown to be locally in $W^{2,2}(\Omega)$. For the case $\gamma > 0$, the authors also obtained $W^{2,2}_{\text{loc}}$ estimates, provided that $|p - 2 + \gamma|$ is sufficiently small.  

More recently, an equation with a nonhomogeneous degeneracy was studied in \cite{WYG24}, where the authors proved local $C^{1, \alpha}$-regularity for viscosity solutions to  
\[
-\left[|Du|^{\alpha(x,u)} + a(x)|Du|^{\beta(x,u)}\right]\NDelta u = f(x) \quad \text{in} \quad B_1,  
\]  
under suitable conditions on the exponents $0 \leq \alpha(\cdot), \beta(\cdot)$ and the coefficient $0 \leq a(\cdot) \in C(B_1)$. Equations with Hamiltonian terms and strong absorptions were also the subject of research. In \cite{BessadaSilva2025}, the authors study the equation
\[
|\nabla u|^\theta\left(\NDelta u + \left\langle {\mathcal B}(x), \nabla u \right\rangle \varrho(x)|\nabla u|^\theta \right) = f(x) \quad \mbox{ in} \quad B_1,
\]
and show a series of properties that include existence of viscosity solutions and gradient estimates for viscosity solutions. While in \cite{AlcdaSilSant2025}, the authors work with the dead core problem
\[
-|\nabla u|^{\theta}\NDelta u = f(x, u) \quad \mbox{ in} \quad B_1, 
\]
and prove in particular that solutions are of class $C^{1, \frac{1}{1+\theta}}$ over the free boundary $\partial\{u=0\}\cap B_{1/2}$. They also establish properties like nondegeneracy of solutions and uniform positive density of the free boundary. 

The primary objective of this paper is to establish regularity results under a general degeneracy law $\sigma(\cdot)$. To the best of our knowledge, the first instance of this type of problem appearing in connection with the regularity of solutions was in the recent work \cite{APPT2022}, where the authors studied a fully nonlinear degenerate elliptic equation of the form  
\begin{equation}\label{eq_degfull}
\sigma(|Du|)F(D^2u) = f(x) \quad \text{in} \quad B_1,    
\end{equation}
where $\sigma(\cdot)$ is a modulus of continuity with $\sigma^{-1}(\cdot)$ being Dini continuous. Under these assumptions, the authors proved that solutions belong to the class $C^1$.  

As usual, the argument relies on the existence of hyperplanes that approximate the solution $u$ near a point $x_0$. The main challenge in applying this method for this particular model is that the sequence of solutions generated in the process does not satisfy the same class of equations. To overcome this difficulty, the authors also introduced a novel technique that prevents the sequence of degeneracy laws constructed in this procedure from further degenerating. However, this approach comes at the cost of losing the $C^{\alpha}$ modulus of continuity for the gradient of solutions. Following this work, the study of regularity results for general degeneracy laws has emerged in various contexts, including nonlocal equations \cite{WF24} and free transmission problems \cite{PS24}. We also refer to \cite{AN25} for further developments. Under similar assumptions, we establish $C^1$-regularity for the model in \eqref{1}. See Theorem \ref{theo_1} below. 

The remainder of this manuscript is organized as follows: In Section 2, we collect some auxiliary results and state our main theorems. In section 3, we establish compactness properties for solutions. Finally, the differentiability of solutions is the subject of Section 4.


\section{Preliminaries}

\subsection{Definitions}

This section gathers some definitions used throughout this article. The notation $B_r(x_0)$ stands for the ball of radius $r$ centered at $x_0$. In what follows, we present the definition of viscosity solutions for the solutions to $-\Delta^N_pu=f$.  

\begin{definition}[Viscosity solution]
Let $1 < p < \infty$ and $f \in C(B_1)$. We say that $u \in C(B_1)$ is a viscosity subsolution to 
\begin{equation}\label{eq_normal}
-\Delta_p^Nu = f(x) \quad \mbox{ in} \quad B_1,    
\end{equation}
if for every $x_0 \in B_1$ and $\varphi \in C^2(B_1)$ such that $u-\varphi$ has a local maximum at $x_0$, then
\begin{equation*}
\begin{cases}
-\Delta_p^N\varphi(x_0) \leq f(x_0), \hspace{3.88cm} \mbox{ if } \quad D\varphi(x_0) \not= 0,\\
-\Delta\varphi(x_0) -(p-2)\lambda_{max}(D^2\varphi(x_0)) \leq f(x_0) \quad \mbox{ if } \quad D\varphi(x_0) = 0 \mbox{ and } p\geq 2,\\
-\Delta\varphi(x_0) -(p-2)\lambda_{min}(D^2\varphi(x_0)) \leq f(x_0) \quad \mbox{ if } \quad D\varphi(x_0) = 0 \mbox{ and } 1 < p < 2.
\end{cases}    
\end{equation*}
Similarly, we say that $u \in C(B_1)$ is a viscosity supersolution to \eqref{eq_normal} if for every $x_0 \in B_1$ and $\varphi \in C^2(B_1)$ such that $u-\varphi$ has a local minimum at $x_0$, then
\begin{equation*}
\begin{cases}
-\Delta_p^N\varphi(x_0) \geq f(x_0), \hspace{3.88cm} \mbox{ if } \quad D\varphi(x_0) \not= 0,\\
-\Delta\varphi(x_0) -(p-2)\lambda_{max}(D^2\varphi(x_0)) \geq f(x_0) \quad \mbox{ if } \quad D\varphi(x_0) = 0 \mbox{ and } p\geq 2,\\
-\Delta\varphi(x_0) -(p-2)\lambda_{min}(D^2\varphi(x_0)) \geq f(x_0) \quad \mbox{ if } \quad D\varphi(x_0) = 0 \mbox{ and } 1 < p < 2.
\end{cases}    
\end{equation*}
A function $u \in C(B_1)$ is a viscosity solution to \eqref{eq_normal} if it is both a viscosity subsolution and a viscosity supersolution. Finally, we say that $u$ is a normalized viscosity solution if $\sup_{B_1}|u| \leq 1$
\end{definition}

\begin{remark}\label{rem_welldef}
Since $\sigma(\cdot)$ is a modulus of continuity, the operator in \eqref{1} is continuous, and we can use the standard notion of viscosity solution, found, for instance, in \cite{crandall1992user}. Indeed, we can rewrite the operator in \eqref{1} as 
\[
F(t, M) := \sigma(|t|)\tr\left[\left(I + v\otimes v\right)M\right],
\]
where $v:=t/|t|$ is a unit vector. Hence, 
\begin{align*}
|\lim_{t \to 0}  F(t, M)| & = \left|\lim_{t \to 0} \left[\sigma(|t|)\tr M + \sigma(|t|)(p-2)\tr(v^TMv)\right]\right| \\
& = \left|\lim_{t \to 0} \left[\sigma(|t|)\tr M + \sigma(|t|)(p-2)v^TMv\right]\right| \\
&\leq \lim_{t \to 0} \left|\sigma(|t|)\tr M\right| + \lim_{t \to 0}|\sigma(|t|)(p-2)|\|M\| \\
&\leq 0.     
\end{align*}
Therefore $\lim_{t \to 0}  F(t, M) = 0$.
\end{remark}

Since the operator in \eqref{1} is continuous, we can use the standard notion of viscosity solution.

\begin{definition}[Viscosity solution II]
Let $1 < p < \infty$ and $f \in C(B_1)$. We say that $u \in C(B_1)$ is a viscosity subsolution to \eqref{1} if for every $x_0 \in B_1$ and $\varphi \in C^2(B_1)$ such that $u-\varphi$ has a local maximum at $x_0$, then
\begin{equation*}
-\sigma(|D\varphi(x_0)|)\Delta_p^N\varphi(x_0) \leq f(x_0).       
\end{equation*}
Similarly, we say that $u \in C(B_1)$ is a viscosity supersolution to \eqref{1} if for every $x_0 \in B_1$ and $\varphi \in C^2(B_1)$ such that $u-\varphi$ has a local minimum at $x_0$, then
\begin{equation*}
-\sigma(|D\varphi(x_0)|)\Delta_p^N\varphi(x_0) \geq f(x_0).       
\end{equation*}
A function $u \in C(B_1)$ is a viscosity solution if it is both a viscosity subsolution and a viscosity supersolution.
\end{definition}

We close this section by letting $c_0$ denote the space of sequences of positive numbers that converge to zero.

\subsection{Assumptions and main results}

In this section, we detail the main assumptions of the paper and state our main results. We start with the conditions satisfied by the degeneracy term $\sigma(\cdot)$.

\begin{Assumption}[Modulus of continuity]\label{Asigma}
We assume the $\sigma(\cdot)$ is a modulus of continuity, i.e., $\sigma : [0, +\infty) \to [0, +\infty)$ is a non-decreasing function with $\displaystyle\lim_{t\to 0}\sigma(t) = 0$. We also assume that $\sigma$ has a normalization, in the sense that $\sigma(1)\geq 1$.
\end{Assumption}

We emphasize that the condition $\sigma(1)\geq 1$ is just a normalization, and it is not a restriction. Next, we give the main assumption on the modulus of continuity $\sigma(\cdot)$

\begin{Assumption}[Dini continuity]\label{Adini}
We suppose that $\sigma^{-1}: \sigma([0, +\infty)) \to [0, +\infty)$ is Dini continuous, meaning that
\[
\int_0^\tau\dfrac{\sigma^{-1}(t)}{t}dt \leq +\infty,
\]
for some $\tau >0$. 
\end{Assumption}
An important characterization of the Dini condition is given in terms of summability along geometric sequences. More precisely, a modulus of continuity $\gamma$ satisfies the Dini condition above if, and only if, for every $\theta\in(0,1)$ 
\[
\sum_{n=1}^{\infty}\gamma(\tau\cdot\theta^n)<\infty.
\]
For more details about this equivalence, we refer the readers to \cite[Definition 1]{APPT2022} and the discussion presented there. The next assumption concerns the regularity of the source term.

\begin{Assumption}[Source term]\label{Af}
We assume that $f \in C(B_1)\cap L^\infty(B_1)$    
\end{Assumption}

Although we assume $f$ to be continuous, it is a mere technicality; none of the estimates presented here will depend on the modulus of continuity of $f$.

Now, we are ready to state the main result of this manuscript, a result concerning the $C^1$-regularity of viscosity solutions to \eqref{1}.

\begin{theorem}\label{theo_1}
Let $u \in C(B_1)$ be a viscosity solution to \eqref{1} and suppose that A\ref{Asigma}-A\ref{Af} are satisfied. Then $u \in C^1_{loc}(B_1)$, and there exists a modulus of continuity $\gamma: [0, +\infty) \to [0, +\infty)$, depending only on $d$, $p$ and $\|f\|_{L^\infty(B_1)}$ such that
\[
|Du(x) - Du(y)| \leq \gamma(|x-y|),
\]
for every $x, y \in B_{1/2}$.
\end{theorem}

Observe that, as in \cite{APPT2022}, we do not assume that the moduli of continuity are necessarily H\"older continuous. But, if in particular we have that $\gamma^{-1}$ is an H\"older continuous function, then we obtain the well-known $C^{1,\alpha}$-regularity of solutions. Hence, Theorem \ref{theo_1} extends, for instance, the works of \cite{Attouchi-Ruosteenoja2020} in the degenerate case. 

\subsection{Auxiliary results}

This section presents a few auxiliary results that are needed in this work. We start by stating a lemma that will help us prove the compactness of solutions.

\begin{lemma}\label{Ishii-Lions-Crandall}
Let $\Omega$ be a locally compact subset of $\mathbb{R}^d$. Let $B\subset B_1$, $u,v\in C(B_1)$ and $\psi$ be twice continuously differentiable in $B\times B$. Define
\[
w(x,y)\coloneqq u(x)+v(y) \quad \forall (x,y)\in B\times B,
\]
and assume that $v-\psi$ attains the maximum at $(x_0,y_0)\in B\times B$. Then for each $\kappa>0$ with $\kappa D^2\psi(x_0,y_0)<I$, there exists $X,Y\in \mathcal{S}(d)$ such that
\[
\left(D_x\psi(x_0,y_0),X\right)\in \overline{\mathcal{J}}^{2,+}u(x_0), \quad \left(D_y\psi(x_0,y_0),Y\right)\in \overline{\mathcal{J}}^{2,+}v(y_0),
\]
Moreover, the block diagonal matrix with entries $X$ and $Y$ satisfies 
\[
-\frac{1}{\kappa}I\leq 
\begin{pmatrix}
X & 0\\
0 & Y
\end{pmatrix} 
\leq \left(I-\kappa D^2\psi(x_0,y_0)\right)^{-1}D^2\psi(x_0,y_0).
\]
\end{lemma}

\begin{remark}[Scaling properties]
Throughout the paper, we assume a smallness condition on certain quantities. We want to reinforce that these conditions are not restrictive. This means that, without loss of generality, we can assume that $u$ is a normalized solution to the equation \eqref{1} and that the source term satisfies a smallness condition. Indeed, if $u\in C(B_1)$ is a bounded viscosity solution to \eqref{1}, then for every $\varepsilon>0$ the function 
\[
w(x)\coloneqq\frac{u(x)}{\|u\|_{L^{\infty}(B_1)}+\varepsilon^{-1}\|f\|_{L^{\infty}(B_1)}},
\]
is such that $\|w\|_{L^{\infty}(B_1)}\leq 1$, and solve
\begin{equation*}
 -\tilde{\sigma}(|Dw(x)|)\NDelta w(x)=\tilde{f}(x)\quad \text{in} \quad B_1,   
\end{equation*}
where 
\[
\tilde{f}(x):=\frac{f(x)}{\kappa } \quad \text{and} \quad \tilde{\sigma}(t)\coloneqq \sigma\left(\kappa t\right),
\]
for $\kappa\coloneqq \|w\|_{L^{\infty}(B_1)}+\varepsilon^{-1}\|f\|_{L^{\infty}(B_1)}$. In particular, we have $\|\tilde{f}\|_{L^{\infty}(B_1)}\leq\varepsilon$. Now, we need to check that $\tilde{\sigma}$ satisfy $A\ref{Asigma}$ and $A\ref{Adini}$. For that, observe
\[
\tilde{\sigma}^{-1}(t)\coloneqq\frac{1}{\kappa }\sigma^{-1}\left(t\right),
\]
since
\[
\tilde{\sigma}^{-1}(\tilde{\sigma}(t))
=
\tilde{\sigma}^{-1}\left(\sigma\left(\kappa t\right)\right)
=
\frac{1}{\kappa }\sigma^{-1}\left(\sigma\left(\kappa t\right)\right)=t.
\]
Moreover, for any $\kappa >1$ it follows that
\[
\int_0^{1}\frac{\tilde{\sigma}^{-1}(t)}{t}dt
=
\int_0^{1}\frac{1}{\kappa }\frac{\sigma^{-1}(t)}{t}dt
\leq
\int_0^{1}\frac{\sigma^{-1}(t)}{t}dt
\]
and once that $\sigma$ is non-decreasing, we obtain
\[
\tilde{\sigma}(1)=\sigma\left(\kappa \right)\geq\sigma(1)\geq 1.
\]
\end{remark}
\subsection{Non-collapsing sets}\label{crafting_sequences}

At this point, for the completeness of the work, we construct a sequence of numbers $(\mu_n)_{n\in \mathbb{N}} \in \N$ and a special family of degeneracy laws ${\mathcal N} = \{\sigma_0, \sigma_1, \sigma_2,\cdots, \sigma_n, \cdots,\}$ that are essential in the proof of Theorem \ref{theo_1}. We begin with a technical lemma, whose proof can be found in \cite[Lemma 1]{APPT2022}. 

\begin{lemma}\label{c_0-in l_1}
Given any sequence of summable numbers $(a_n)_{n\in\mathbb{N}}\in \ell_1$ and $\varepsilon,\delta>0$, there exists a sequence $(c_n)_{n\in\mathbb{N}} \in c_0$, such that
\[
(b_n)_{n\in\mathbb{N}}:=\left(\frac{a_n}{c_n}\right)_{n\in\mathbb{N}}\in \ell_1, \quad \max_{n\in\mathbb{N}}|c_n|\leq \frac{1}{\varepsilon}
\]
and
\[
\varepsilon\left(1-\frac{\delta}{2}\right)\sum_{n=1}^{\infty}|a_n|\leq\sum_{n=1}^{\infty}|b_n|\leq\varepsilon\left(1+\delta\right)\sum_{n=1}^{\infty}|a_n|.
\]
\end{lemma}

As shown in \cite{APPT2022}, the Dini continuity of $\sigma^{-1}$ plays a key role in establishing the $C^1$-regularity of viscosity solutions to \eqref{1}. This condition makes it possible to construct the so-called {\it non-collapsing sets} of moduli of continuity, a concept introduced in \cite{APPT2022} to analyse the regularity properties for viscosity solutions of fully nonlinear elliptic equations such as \eqref{eq_degfull}. In what follows, we recall some relevant definitions and additional properties related to non-collapsing sets.

\begin{definition}[\cite{APPT2022}, Definition 3]\label{def_noncolapsing}
We say that a set $\Xi$ of moduli of continuity defined over an interval $I \in {\mathcal I}$ is a non-collapsing set if for all sequences $(g_n)_{n \in \N} \subset \Xi$, and all scalars $(a_n)_{n \in \N} \subset I$, we have
\[
g_n(a_n) \to 0 \quad \mbox{implies} \quad a_n \to 0.
\]
The set ${\mathcal I}$ is defined as ${\mathcal I} := \{(0, T)\,|\, 0< T< \infty\}\cup\{\R_0^+\}$.
\end{definition}

\begin{definition}[\cite{APPT2022}, Definition 5] We say that a sequence of moduli of continuity $(\sigma_n)_{n \in N}$ is shored-up if we can find a sequence of positive numbers $(m_n)_{n \in N}$ such that $m_n \to 0$ and
\[
\inf_n\sigma_n(m_n) >0 \quad \mbox{ for all} \quad n \in \N.
\]
\end{definition}

\begin{proposition}[\cite{APPT2022}, Proposition 5]\label{prop_sn}
If a sequence of moduli of continuity $(\sigma_n)_{n \in \N}$ is shored-up, then $\Xi = \cup_{n \in \N}\{\sigma_n\}$ is non-collapsing.
\end{proposition}

Now, let us begin with the construct of $(\mu_n)_{n\in\mathbb{N}}$, for this, consider $C>0$ and $\beta\in (0,1)$ be universal constants and define 
\[
\gamma(t)\coloneqq t\sigma(t) \quad \text{and} \quad \omega(t)\coloneqq\gamma^{-1}(t).
\]
Next, we choose the first term of the sequence by selecting $0<\rho<\mu_1<1$ according to the following alternatives:

\begin{itemize}
\item[i)] If $\left(t^{\beta}/\omega(t)\right)\to 0$ as $t\to\infty$, we take $0<\rho\ll 1$ such that
\[
2C\rho^{\beta}=\omega(\rho)\coloneqq\mu_1>\rho.
\]

\item[ii)] If the degeneration law is stronger than $t^{\frac{1}{\beta}-1}$, which means, $\omega(t)=O(t^{\beta})$, then for a fixed $0<\alpha<\beta$, we take $0<\rho<1$ sufficiently small such that
\[
2C\rho^{\beta}=\rho^{\alpha}=\colon \mu_1>\rho.
\]
\end{itemize}

In any case,  we define the ratio 
\[
0<\theta\coloneqq\frac{\rho}{\mu_1}<1.
\]
From the assumption $A\ref{Adini}$, the sequence 
\[
(a_n)_{n\in\mathbb{N}}\coloneqq\left(\sigma^{-1}(\theta^{n})\right)_{n\in\mathbb{N}}\in \ell^1.
\]
Applying the Lemma \ref{c_0-in l_1} to $(a_n)_{n\in \mathbb{N}}$ with $0<\delta<1$ and $\varepsilon\coloneqq 1/(1+\delta)$, we are able to find a sequence $(c_n)_{n\in\mathbb{N}} \in c_0$ of positive numbers satisfying 
\[
\varepsilon\left(1-\frac{\delta}{2}\right)\sum_{i=n}^{\infty}\sigma^{-1}(\theta^n)\leq\sum_{n=1}^{\infty}\frac{\sigma^{-1}(\theta^n)}{c_n}\leq\sum_{n=1}^{\infty}\sigma^{-1}(\theta^n).
\]
At this point, using $(c_n)_{n\in\mathbb{N}}$, we construct the family $\mathcal{N}\coloneqq\{\sigma_0(t),\sigma_1(t),...,\sigma_n(t),...\}$ of modulus of continuity given by the rule:
\[
\sigma_n(t)=\frac{\prod_{j=1}^{n}\mu_j}{r^n}\sigma\left(\left(\prod_{j=1}^n\mu_j\right)\cdot t\right),
\]
where $\mu_1$ is as above, and for $n\geq 2$ the value $\mu_{n}$ is chosen  $r<\mu_1\leq\mu_2\leq\cdots\leq\mu_n$ according to one of the following conditions:
\begin{itemize}
    \item[i)] Either $\mu_n>\mu_{n-1}$ is chosen such that $\sigma_{n}(c_n)=1$;
    \item[ii)] Or $\mu_n=\mu_{n-1}$ is such that $\sigma_{n}(c_n)\geq 1$. 
\end{itemize}
From Proposition \ref{prop_sn}, we conclude that ${\mathcal N}$ is a non-collapsing set.

\section{Compactness of the solutions}

In this section, we prove a compactness result of the solutions to \eqref{1}. In fact, since our arguments involve scaling, we need to show such a property for a larger class of equations. Namely, we work with solutions to

\begin{equation}\label{eq_scal}
-\sigma(|Du+\xi|)\left[\Delta u +(p-2)\left\langle D^2u\dfrac{Du + \xi}{|Du + \xi|}, \dfrac{Du + \xi}{|Du + \xi|} \right\rangle\right] = f(x) \quad \mbox{ in } \quad B_1, 
\end{equation}
where $\xi \in \R^d$. We begin by proving a couple of auxiliary lemmas that will help us to establish the compactness result.

\begin{lemma}\label{lem_compac1}
Let $u\in C(B_1)$ be a normalized viscosity solution to \eqref{eq_scal}, with $\|f\|_{L^{\infty}(B_1)}\leq 1$. There exists $\Gamma>0$ such that if $|\xi|>\Gamma$, then $u$ belongs to  $C^{\beta}(B_{1/2})$ for some universal $\beta\in(0,1)$. Moreover, there exists a constant $C>0$ such that 
\begin{equation*}
\|u\|_{C^{\beta}(1/2)}\leq C.
\end{equation*}
\end{lemma}

\begin{proof}
Let us start by fixing $\beta\in(0,1)$ and defining the modulus of continuity 
$\omega:\mathbb{R}^+\longrightarrow\mathbb{R}^+$ given by $\omega(t):=t^{\beta}$. Let $0<r\ll 1$, and for a fixed $z\in B_{r/2}$ define the constant
\[
\mathbf{M}:=\sup_{x,y\in B_r}\left(u(x)-u(y)-C_1\omega(|x-y|)-C_2(|x-z|^2+|y-z|^2)\right).
\]
Our goal is to show that $\mathbf{M}\leq 0$ for some constants $C_1, C_2>0$. We will argue by contradiction; for this reason, suppose that for every $C_1>0$ and $C_2>0$, there exists a point $z\in B_{r/2}$ for which $\mathbf{M}>0$. \\

\noindent{\bf Step 1:} {\it Choosing the constant $C_2$}.\\
Consider the following auxiliary functions 
\[
\psi(x,y)\coloneqq C_1\omega(|x-y|)+C_2(|x-z|^2+|y-z|^2) \quad \text{and} \quad \varphi(x,y):= u(x)-u(y)-\psi(x,y).
\]
Let $(x_0,y_0) \in \bar{B}_r\times\bar{B}_r$ be such that $\max_{\bar{B}_r\times\bar{B}_r} \varphi(x, y) = \varphi(x_0, y_0)$ and notice that $x_0\neq y_0$. In fact, if $x_0=y_0$, we would have
\[
\mathbf{M}\leq 0,
\]
for every positive constant $C_2$. Hence, 
\[
\varphi(x_0,y_0)=\max_{\bar{B}_r\times\bar{B}_r}\varphi(x,y)=\mathbf{M}>0.
\]
Since $u$ is a normalized solution, we obtain
\[
\psi(x_0,y_0) < u(x_0)-u(y_0) \leq \osc_{B_1}u \leq 2,
\]
which in particular implies
\[
C_1\omega(|x_0-y_0|)+C_2(|x_0-z|^2+|y_0-z|^2).
\]
We now choose the constant $C_2$ to ensure that $(x_0,y_0)\in B_r\times B_r$. For this purpose, it is sufficient to take $C_2\coloneqq\left(32/r\right)^2$. Indeed notice that
\[
 \left(\frac{32}{r}\right)^2|x_0-z|^2
 =
 C_2|x_0-z|^2\leq 2
 \quad \text{and} \quad
\left(\frac{32}{r}\right)^2|y_0-z|^2
=
C_2|y_0-z|^2\leq 2,
\]
which implies
\begin{equation}\label{r/16}
|x_0-z|\leq\frac{r}{16} \quad \text{and} \quad |y_0-z|\leq \frac{r}{16}.
\end{equation}

\noindent{\bf Step 2:} {\it Applying Lemma \ref{Ishii-Lions-Crandall}}. \\
We compute 
\[
 D_{x}\psi(x_0,y_0)=C_1\omega'(|x_0-y_0|)\frac{x_0-y_0}{|x_0-y_0|}+2C_2(x_0-z)
\]
and
\[
-D_{y}\psi(x_0,y_0)=C_1\omega'(|x_0-y_0|)\frac{x_0-y_0}{|x_0-y_0|}-2C_2(y_0-z).
\]
From Lemma \ref{Ishii-Lions-Crandall}, we ensure the existence of $X,Y\in \mathcal{S}(d)$ such that 
\begin{equation}\label{block inequality}
 -\frac{2}{\kappa}
\begin{pmatrix}
I & 0\\
0 & I
\end{pmatrix} 
\leq 
\begin{pmatrix}
X & 0\\
0 & -Y
\end{pmatrix} 
\leq 
\begin{pmatrix}
B^{\kappa} & -B^{\kappa}\\
-B^{\kappa} & B^{\kappa}
\end{pmatrix},
\end{equation}
for each $\kappa>0$ such that $\kappa B<I$, where the matrix $B$ is given by (recall that $\omega(t) = t^\beta$) 
\begin{align*}
B
&=
C_1\beta|x_0-y_0|^{\beta-2}\left(I+(\beta-2)\frac{x_0-y_0}{|x_0-y_0|}\otimes\frac{x_0-y_0}{|x_0-y_0|}\right).
\end{align*}
Choosing $\kappa:=1/(2C_1\beta|x_0-y_0|^{\beta-2})>0$, we get
\[
I-\kappa B=\frac{1}{2}I-\frac{(\beta-2)}{2}\frac{x_0-y_0}{|x_0-y_0|}\otimes\frac{x_0-y_0}{|x_0-y_0|},
\]
and knowing that $x\otimes x=x\cdot x^T$, by Sherman-Morrison formula we obtain 
\begin{align*}
\left(I-\kappa B\right)^{-1}
&=
2\left[I-\frac{I\cdot\left(-\frac{\beta-2}{2}\left(\frac{x_0-y_0}{|x_0-y_0|}\right)\cdot\left(\frac{x_0-y_0}{|x_0-y_0|}\right)^{T}\right)\cdot2I}{1+\left(\frac{x_0-y_0}{|x_0-y_0|}\right)^{T}\cdot2I\cdot\left(-\frac{(\beta-2)}{2}\frac{x_0-y_0}{|x_0-y_0|}\right)}\right]\\
&=
2\left[I-\frac{(2-\beta)}{(3-\beta)}\left(\frac{x_0-y_0}{|x_0-y_0|}\right)\cdot\left(\frac{x_0-y_0}{|x_0-y_0|}\right)^{T}\right].
\end{align*}
As a consequence,
\[
B^{\kappa}=(I-\kappa B)^{-1}B= 2C_1\beta|x_0-y_0|^{\beta-2}\left(I-2\frac{2-\beta}{3-\beta}\frac{x_0-y_0}{|x_0-y_0|}\otimes\frac{x_0-y_0}{|x_0-y_0|}\right).
\]

\noindent{\bf Step 3:} {\it Recovering the information on the eigenvalues of $X-Y$}.\\
First, let us take any vector $(\xi,\xi)\in \mathbb{R}^{2d}$ with $|\xi|=1$, and observe that by \eqref{block inequality} we have
\[
\begin{pmatrix}
\xi\\
\xi
\end{pmatrix}^{T}
\begin{pmatrix}
X & 0\\
0 & -Y
\end{pmatrix}
\begin{pmatrix}
\xi\\
\xi
\end{pmatrix}
\leq 
\begin{pmatrix}
\xi\\
\xi
\end{pmatrix}^{T}
\begin{pmatrix}
B^{\kappa} & -B^{\kappa}\\
-B^{\kappa} & B^{\kappa}
\end{pmatrix}
\begin{pmatrix}
\xi\\
\xi
\end{pmatrix}.
\]
Hence, $X-Y\leq 0$, and since
\begin{equation}\label{Inner_product_B}
\langle B^{\kappa} \xi, \xi \rangle = 2C_1\beta|x_0-y_0|^{\beta-2}\left(\frac{\beta-1}{3-\beta}\right)<0,
\end{equation}
we obtain
\begin{equation}\label{norma_of_X_Y}
    \|X\|,\|Y\|\leq 2C_1\beta|x_0-y_0|^{\beta-2}.
\end{equation}
Moreover, applying the estimate above to the vector $(\xi,-\xi)$, where $\xi\coloneqq(x_0-y_0)/|x_0-y_0|$ and using the previous estimate, we get
\begin{equation}\label{negative_eigenvaleu}
\langle(X-Y)\xi,\xi\rangle\leq 4\langle B^{\kappa}\xi,\xi\rangle=8C_1\beta|x_0-y_0|^{\beta-1}\left(\frac{\beta-1}{3-\beta}\right)<0,
\end{equation}
which implies that at least one eigenvalue of $X-Y$ is negative and smaller than the constant above. Throughout the remainder of the proof, we will denote this eigenvalue as $\lambda_0$.\\

\noindent{\bf Step 4:} {\it The inequality in the viscosity sense and the choice of $\Gamma$}.\\
For simplicity, let us denote 
\[
\xi_{x_0}\coloneqq D_{x}\psi(x_0,y_0) \quad \text{and} \quad \xi_{y_0}\coloneqq D_{y}\psi(x_0,y_0).
\]
Now, fix an arbitrary vector $\xi\in\mathbb{R}^d$, take
\[
\Gamma\coloneqq\left(\frac{9}{4}\right)C_1\beta|x_0-y_0|^{\beta-1},
\]
and consider the scenario where $|\xi|>\Gamma$. Observe that since $|x_0-y_0|\leq 2$, we can choose the constant $C_1>\left(2^{9-\beta}/r\beta\right)$ so that
\[
C_1>\frac{2^{9-\beta}}{r\beta}=\frac{2^{8}}{r\beta 2^{\beta-1}}\geq \frac{2^8}{r\beta|x_0-y_0|^{\beta-1}}.
\]
Hence, using the information in \eqref{r/16} and the definition of $C_2$, we obtain
\begin{align*}
\frac{C_1}{2}\beta|x_0-y_0|^{\beta-1}&>\frac{2^7}{r} =\left(\frac{32}{r}\right)^2\frac{r}{8}\\
&\geq
2C_2|x_0-z_0|.
\end{align*}
As a consequence, we estimate
\begin{equation*}\label{controle de xi_x0}
2C_1\beta|x_0-y_0|^{\beta-1}\geq|\xi_{x_0}|\geq\frac{C_1}{2}\beta|x_0-y_0|^{\beta-1},
\end{equation*}
\begin{equation*}\label{controle de xi_yo}
2C_1\beta|x_0-y_0|^{\beta-1}\geq|\xi_{y_0}|\geq\frac{C_1}{2}\beta|x_0-y_0|^{\beta-1}. 
\end{equation*}

Therefore, the choice that we make for $\Gamma$ allows us to conclude

\begin{equation}\label{controle de nu_1}
|\xi_{x_0}+\xi|\geq\frac{C_1}{4}\beta|x_0-y_0|^{\beta-1},
\end{equation}
\begin{equation}\label{controle de nu_2}
|\xi_{y_0}+\xi|\geq\frac{C_1}{4}\beta|x_0-y_0|^{\beta-1}.
\end{equation}
Consider $\nu_1\coloneqq\xi_{x_0}+\xi$ and $\nu_2\coloneqq\xi_{y_0}+\xi$. Since $u$ is a viscosity solution to \eqref{eq_scal}, we obtain the following inequalities 
\[
\sigma(|\nu_1|)\left[\tr(X+C_2I)+(p-2)\dfrac{\langle(X+C_2I)(\nu_1),(\nu_1)\rangle}{|\nu_1|^2}\right]+f(x_0)\geq 0
\]
and
\[
\sigma(|\nu_2|)\left[\tr(Y-C_2I)+(p-2)\dfrac{\langle(Y-C_2I)(\nu_2),(\nu_2)\rangle}{|\nu_2|^2}\right]+f(y_0)\leq 0,
\]
which implies,
\begin{equation}\label{ineq12}
\tr\left(A(\nu_1)(X+C_2I)\right)\geq-\frac{f(x_0)}{\sigma(|\nu_1|)}
\quad \text{and}\quad 
-\tr\left(A(\nu_2)(Y-C_2I)\right)\geq\frac{f(y_0)}{\sigma(|\nu_2|)},
\end{equation}
where for $\nu\neq 0$, $\overline{\nu}=\nu/|\nu|$ and
\[
A(\nu)\coloneqq I+(p-2)\overline{\nu}\otimes\overline{\nu}.
\]

\noindent{\bf Step 5:} {\it Estimates for the viscosity inequality}.\\
Putting together the inequalities \eqref{ineq12}, the fact that $\sigma(t)$ is non-decreasing, and $f\in L^{\infty}(B_1)$, we obtain
\begin{align*}
\tr\left(A(\nu_1)(X+C_2I)\right)-\tr\left(A(\nu_2)(Y-C_2I)\right)
&\geq
\frac{f(y_0)}{\sigma(|\nu_2|)}-\frac{f(x_0)}{\sigma(|\nu_1|)}\\
&\geq
-\dfrac{\|f\|_{L^{\infty}(B_1)}}{\sigma\left(|\frac{C_1}{4}\beta|x_0-y_0|^{\beta-1}\right)}-\dfrac{\|f\|_{L^{\infty}(B_1)}}{\sigma\left(|\frac{C_1}{4}\beta|x_0-y_0|^{\beta-1}\right)}\\
&\geq
-\dfrac{2\|f\|_{L^{\infty}(B_1)}}{\sigma(1)}\\
&\geq
-2\|f\|_{L^{\infty}(B_1)}.
\end{align*}
On the other hand, using the linearity of the Trace operator, we obtain
\begin{align*}
\tr\left(A(\nu_1)(X+C_2I)\right) - \tr\left(A(\nu_2)(Y-C_2I)\right) 
&=
\tr(A(\nu_1)X) + C_2\tr(A(\nu_1)I) - 
\left[\tr(A(\nu_2)Y) - C_2\tr(A(\nu_2)I)\right] \\
&= 
\tr(A(\nu_1)(X-Y))+C_2\left[\tr(A(\nu_1)) + \tr(A(\nu_2))\right]\\
&+
\tr\left(\left(A(\nu_1)-A(\nu_2)\right)Y\right)\\
&\coloneqq
(I)+(II)+(III).
\end{align*}
At this point, we analyse the right-hand side of the inequality above. We begin estimating $(I)$, and for that remember that the eigenvalues of $A(\nu_1)$ belong to the interval $\left[\min\{1,p-1\},\max\{1,p-1\}\right]$, and \eqref{negative_eigenvaleu} hold, so it follows that
\begin{align*}
\tr(A(\nu_1)(X-Y))&
\leq
\sum_{n=1}^d\lambda_{n}(A(\nu_1))\lambda_{n}(X-Y)\leq\min\{1,p-1\}\lambda_{0}(X-Y)\\
&\leq
8C_1\beta\min\{1,p-1\}|x_0-y_0|^{\beta-1}\left(\frac{\beta-1}{3-\beta}\right).
\end{align*}
To do the estimative of $(II)$, it is sufficiently see that  
\begin{align*}
C_2\left[\tr(A(\nu_1))+\tr(A(\nu_2))\right]&\leq C_2\left(\sum_{n=1}^d\lambda_{n}(A(\nu_1))+\sum_{n=1}^d\lambda_{n}(A(\nu_2))\right)\\
&\leq
2dC_2\max\{1,p-1\}.
\end{align*}
For the estimate of $(III)$, observe that from the definition of $A(\nu)$ we can write 
\begin{align*}
A(\nu_1)-A(\nu_2)&=(p-2)\left(\overline{\nu}_1\otimes\overline{\nu}_1-\overline{\nu}_2\otimes\overline{\nu}_2+\overline{\nu}_2\otimes\overline{\nu}_1-\overline{\nu}_2\otimes\overline{\nu}_1\right)\\
&=
(p-2)\left[(\overline{\nu}_1-\overline{\nu}_2)\otimes\overline{\nu}_1-\overline{\nu}_2\otimes(\overline{\nu}_2-\overline{\nu}_1)\right],
\end{align*}
and as a consequence, we use the information in \eqref{norma_of_X_Y} together with the fact that $|\overline{\nu}_i|=1$, to get
\begin{align*}
\tr(\left(A(\nu_1)-A(\nu_2)\right)Y)&\leq d\|Y\|\cdot\|A(\nu_1)-A(\nu_2)\|
\leq
d|p-2|\cdot\|Y\|\cdot|\overline{\nu}_1-\overline{\nu}_2|(|\overline{\nu}_1|+|\overline{\nu}_2|)\\
&\leq
2d|p-2|\cdot\|Y\|\cdot|\overline{\nu}_1-\overline{\nu}_2|\\
&\leq
4dC_1\beta|p-2|\cdot|x_0-y_0|^{\beta-2}\cdot|\overline{\nu}_1-\overline{\nu}_2|.
\end{align*}
Note that from the definitions of $\xi_{x_0}$ and $\xi_{y_0}$, we know that
\[
|\nu_1-\nu_2|=|\xi_{x_0}-\xi_{y_0}|=2C_2|(x_0-z)+(y_0-z)|\leq\frac{C_2}{4},
\]
hence, combining of the inequality above with \eqref{controle de nu_1} and \eqref{controle de nu_2} yields to
\begin{align*}
|\overline{\nu}_1-\overline{\nu}_2|&=\left(\frac{\nu_1}{|\nu_1|}-\frac{\nu_2}{|\nu_2|}\right)\leq 2\max\left\{\frac{|\nu_1-\nu_2|}{|\nu_1|},\frac{|\nu_1-\nu_2|}{|\nu_2|}\right\}\\
&\leq
2\left(\frac{|\nu_1-\nu_2|}{|\nu_1|}+\frac{|\nu_1-\nu_2|}{|\nu_2|}\right)\\
&\leq 
C_2\left(\frac{1}{|\nu_1|}+\frac{1}{|\nu_2|}\right)\\
&\leq
\frac{8C_2}{C_1\beta|x_0-y_0|^{\beta-1}}.
\end{align*}
Putting together the inequality with the last estimate, we obtain 
\[
\tr(\left(A(\nu_1)-A(\nu_2)\right)Y)\leq 32dC_2|p-2|\cdot|x_0-y_0|^{-1}.
\]
Therefore, 
\begin{align*}
\tr\left(A(\nu_1)(X+C_2I)\right) - \tr\left(A(\nu_2)(Y-C_2I)\right)
&\leq
(I)+(II)+(III)\\
&\leq
8C_1\beta\min\{1,p-1\}|x_0-y_0|^{\beta-1}\left(\frac{\beta-1}{3-\beta}\right)\\
&+
2dC_2\max\{1,p-1\}+32dC_2|p-2|\cdot|x_0-y_0|^{-1}.
\end{align*}

\noindent{\bf Step 6:} {Choosing the constant $C_1$}.\\
From Step 5, we get
\[
0\leq 8C_1\beta\min\{1,p-1\}|x_0-y_0|^{\beta-1}\left(\frac{\beta-1}{3-\beta}\right)+2\|f\|_{L^{\infty}(B_1)}+2dC_2\max\{1,p-1\}+32dC_2|p-2|\cdot|x_0-y_0|^{-1}.
\]
On the other hand, since $\beta<1$, the first term of the inequality above is negative, while the remaining terms are positive; as a consequence, choosing $C_1$ sufficiently large, i.e., 

\[
C_1>\frac{(3-\beta)\left[2\|f\|_{L^{\infty}(B_1)}+2dC_2\left(\max\{1,p-1\}+32|p-2||x_0-y_0|^{-1}\right)\right]}{4(1-\beta)\min\{1,p-1\}|x_0-y_0|^{\beta-1}},
\]
we obtain 
\[
8C_1\beta\min\{1,p-1\}|x_0-y_0|^{\beta-1}\left(\frac{\beta-1}{3-\beta}\right)+2\|f\|_{L^{\infty}(B_1)}+2dC_2\max\{1,p-1\}+32dC_2|p-2|\cdot|x_0-y_0|^{-1}<0,
\]
which is a contradiction. Therefore, $\mathbf{M}\leq 0$ for universal constants $C_1, C_2>0$, and the desired result follows, i.e., there exists a universal constant $C>0$ such that 
\[
|u(x)-u(y)|\leq C|x-y|^{\beta}.
\]
\end{proof}

Once the value $\Gamma$ is fixed, we now treat the case where $|\xi| \leq \Gamma$.

\begin{lemma}\label{lem_compac2}
Let $u \in C(B_1)$ be a normalized viscosity solution to \eqref{eq_scal}, with $\|f\|_{L^\infty(B_1)} \leq 1$. Suppose further that $|\xi| \leq \Gamma$. Then, there exist universal constants $\beta \in (0, 1)$ and $C> 0$ such that 
\[
\|u\|_{C^\beta(B_{3/4})} \leq C.
\]
\end{lemma}

\begin{proof}
Notice that if $|Du| \geq 2\Gamma$, then
\[
|Du + \xi| \geq ||Du| - |\xi|| \geq \Gamma \geq 1,
\]
which implies $\sigma(|Du + \xi|) \geq 1$. Hence the operator in \eqref{eq_scal} is uniformly elliptic, and $u$ also satisfies
\begin{equation*}
\begin{cases}
{\mathcal M}^+(D^2u) + |f| \geq 0 \quad \mbox{ in } \quad B_1 \\
{\mathcal M}^-(D^2u) - |f| \leq 0 \quad \mbox{ in } \quad B_1,
\end{cases}    
\end{equation*}
in the viscosity sense. Therefore, from \cite{imb_silv_jems}, we can infer the existence of a universal constant $\beta \in (0, 1)$, such that $u\in C^\beta(B_{3/4})$ with the estimate
\[
\|u\|_{C^\beta(B_{3/4})} \leq C,
\]
where $C$ is a positive universal constant.
\end{proof}
\begin{proposition}[Hölder continuity]\label{local_holder_continuity}
Let $u\in C(B_1)$ be a normalized viscosity solution to $\eqref{eq_scal}$, with $\|f\|_{L^\infty(B_1)} \leq 1$ and $\xi\in\R^d$ arbitrary. Then, $u$ is locally Hölder continuous in $B_1$. Moreover, there exist universal constants $C>0$ and $\beta \in (0,1)$ such that 
\[
\sup_{\substack{x,y \in B_{3/4} \\ x \neq y}} \frac{|u(x) - u(y)|}{|x - y|^{\beta}}\leq C.
\]
\end{proposition}

\begin{proof}
The proof follows from a direct combination of the Lemmas \ref{lem_compac1} and \ref{lem_compac2}.
\end{proof}

In what follows, we present a property known as the cancellation law, which states that a solution to the degenerate homogeneous problems is also a solution to the elliptic homogeneous equation.

\begin{proposition}[Cancellation law]\label{prop_calaw}
Let $u \in C(B_1)$ be a viscosity solution to
\[
-\sigma(|Du+\xi|)\left[\Delta u +(p-2)\left\langle D^2u\dfrac{Du + \xi}{|Du + \xi|}, \dfrac{Du + \xi}{|Du + \xi|} \right\rangle\right] = 0 \quad \mbox{ in } \quad B_1,
\]
where $\xi \in \R^d$. Then $u$ is a viscosity solution to
\begin{equation}\label{eq_ellhom}
-\Delta u -(p-2)\left\langle D^2u\dfrac{Du + \xi}{|Du + \xi|}, \dfrac{Du + \xi}{|Du + \xi|} \right\rangle = 0 \quad \mbox{ in } \quad B_1.     
\end{equation}
\end{proposition}

\begin{proof}
We set $v(x) := u + \xi\cdot x$, so that $v$ is a viscosity solution to
\[
-\sigma(|Dv|)\left[\Delta v + (p-2)\left\langle D^2v\dfrac{Dv}{|Dv|}, \dfrac{Dv}{|Dv|} \right\rangle\right] = 0 \quad \mbox{ in} \quad B_1.
\]
Let $x_0 \in B_1$ and $\varphi \in C^2(B_1)$ be a function that touches $v$ from below at $x_0$ and $D\varphi(x_0) \not= 0$. Since $v$ is a viscosity solution to the equation above, we have
\[
-\sigma(|D\varphi(x_0)|)\left[\Delta \varphi(x_0) + (p-2)\left\langle D^2\varphi(x_0)\dfrac{D\varphi(x_0)}{|D\varphi(x_0)|}, \dfrac{D\varphi(x_0)}{|D\varphi(x_0)|} \right\rangle\right] \leq 0 \quad \mbox{ in} \quad B_1,
\]
which implies
\[
-\Delta \varphi(x_0) - (p-2)\left\langle D^2\varphi(x_0)\dfrac{D\varphi(x_0)}{|D\varphi(x_0)|}, \dfrac{D\varphi(x_0)}{|D\varphi(x_0)|} \right\rangle \leq 0 \quad \mbox{ in} \quad B_1.
\]
Recall that from \cite{jlm, KMP12}, in the homogeneous case, it is enough to use test functions with $D\varphi(x_0) \not=0$, hence $v$ is a viscosity supersolution to
\[
-\Delta v - (p-2)\left\langle D^2v\dfrac{Dv}{|Dv|}, \dfrac{Dv}{|Dv|} \right\rangle = 0 \quad \mbox{ in} \quad B_1,
\]
Rescaling back to $u$ we obtain
\[
-\Delta u -(p-2)\left\langle D^2u\dfrac{Du + \xi}{|Du + \xi|}, \dfrac{Du + \xi}{|Du + \xi|} \right\rangle \leq 0 \quad \mbox{ in } \quad B_1, 
\]
in the viscosity sense. The case for supersolutions is similar, therefore, $u$ is a viscosity solution to \eqref{eq_ellhom}.
\end{proof}

\begin{proposition}[Stability]\label{prop_stab}
Let $(u_n)_{n \in \N}$ be a sequence of viscosity solutions to 
\begin{equation}\label{eq_solseq2}
-\sigma_n(|Du_n+\xi_n|)\left[\Delta u_n +(p-2)\left\langle D^2u_n \dfrac{Du_n + \xi_n}{|Du_n + \xi_n|}, \dfrac{Du_n + \xi_n}{|Du_n + \xi_n|} \right\rangle\right] = f_n(x) \quad \mbox{ in } \quad B_1.     
\end{equation}
Assume that $A\ref{Asigma}$ - $A\ref{Af}$ hold true and that $(\sigma_n)_{n\in\N} \subset {\mathcal N}$. Suppose further that there exists $u_\infty \in C(B_1)$ such that $u_n \to u_\infty$ locally uniformly in $B_1$ and that $\|f_n\|_{L^\infty(B_1)} \to 0$. We have
\begin{itemize}
\item[1)] If $(\xi_n)_{n\in\N}$ is bounded, then $u_\infty$ solves
\begin{equation}\label{eq_scallim1}
-\Delta u_\infty -(p-2)\left\langle D^2u_\infty \dfrac{Du_\infty + \xi_\infty}{|Du_\infty + \xi_\infty|}, \dfrac{Du_\infty + \xi_\infty}{|Du_\infty + \xi_\infty|} \right\rangle = 0 \quad \mbox{ in } \quad B_{4/5}, 
\end{equation}
in the viscosity sense, for some $\xi_\infty \in \R^d$.
\item[2)] If $(\xi_n)_{n\in\N}$ is unbounded, then $u_\infty$ solves 
\begin{equation}\label{eq_scallim2}
-\Delta u_\infty -(p-2)\left\langle D^2 u_\infty e_\infty, e_\infty \right\rangle = 0 \quad \mbox{ in } \quad B_{4/5},
\end{equation}
in the viscosity sense, where $|e_\infty| = 1$.
\end{itemize}
\end{proposition}

\begin{proof}
Consider $\varphi$ a quadratic polynomial touching $u_\infty$ from below at $x_0 \in B_{4/5}$. Without loss of generality, we can assume that $\varphi$ has the form
\[
\varphi(x) = u_\infty(x_0) + b\cdot(x-x_0) + \dfrac{1}{2}\left\langle M(x-x_0), (x-x_0) \right\rangle,
\]
For simplicity, we can also assume that $|x_0| = u_\infty(x_0) = 0$. Since $u_n$ converges to $u_\infty$, we can find a sequence of points $(x_n)_{n \in \mathbb{N}}$ and quadratic polynomials $(\varphi_n)_{n \in \mathbb{N}}$ of the form
\[
\varphi_n(x) = u_\infty(x_n) + b\cdot(x-x_n) + \dfrac{1}{2}\left\langle M(x-x_n), (x-x_n) \right\rangle,
\]
such that $x_n \to x_0$, $\varphi_n \to \varphi$ locally uniformly, and $\varphi_n$ touches $u_n$ from below at $x_n$.

\medskip

\noindent{\bf Case 1)} Suppose that the sequence $(\xi_n)_{n\in\mathbb{N}}$ is bounded. Then, up to a subsequence, we have that $\xi_n \to \xi_\infty$ in $B_{4/5}$. Once, in the homogeneous case, it is sufficient to use test functions with non-zero gradient, we can assume that
\[
|b+\xi_{\infty}|>0.
\]
Since $(\sigma_{n})_{n \in \N} \subset {\mathcal N}$, we have that $\sigma_n(b+\xi_n) \not\to 0$ (recall Definition \ref{def_noncolapsing}). Therefore,  since $u_n$ is a viscosity solution to \eqref{eq_solseq2}, we have
\[
-\left[\Delta \varphi_{n}(x_n)+(p-2)\left\langle M\frac{b+\xi_{n}}{|b+\xi_{n}|},\frac{b+\xi_{n}}{|b+\xi_{n}|}\right\rangle\right]\leq\frac{f_n(x_n)}{\sigma_n(|b+\xi_n|)} \leq \frac{\|f_n\|_{L^\infty(B_1)}}{\sigma_n(|b+\xi_n|)}\quad \mbox{ in } \quad B_1, 
\]
and by passing the limit as $n \to \infty$, we obtain
\[
-\Delta \varphi(x_0) -(p-2)\left\langle M \dfrac{b + \xi_\infty}{|b + \xi_\infty|}, \dfrac{b + \xi_\infty}{|b + \xi_\infty|} \right\rangle \leq 0 \quad \mbox{ in } \quad B_1,  
\]
which implies that $u_\infty$ is a viscosity subsolution to \eqref{eq_scallim1}. A similar argument shows that $u_\infty$ is a viscosity supersolution, and hence a viscosity solution to \eqref{eq_scallim1}.

\medskip

\noindent{\bf Case 2)} Now, suppose that $(\xi_n)_{n\in\mathbb{N}}$ is unbounded. In this case, we consider the normalized sequence $e_n := \xi_n/|\xi_n|$, which, up to a subsequence, converges to $e_\infty$ in $B_{4/5}$. Notice that since $|\xi_n| \to \infty$, we have, for $n$ sufficiently large, that $b|\xi_n|^{-1} \not= -e_n$ and $|b + \xi_n| > 1$, which implies $\sigma_n(|b+\xi_n|) > 1$. Hence, since $u_n$ is a viscosity solution to \eqref{eq_solseq2}, we have
\[
-\Delta \varphi_n(x_n) -(p-2)\left\langle M \dfrac{b|\xi_n|^{-1} + e_n}{|b|\xi_n|^{-1} + e_n|}, \dfrac{b|\xi_n|^{-1} + e_n}{|b|\xi_n|^{-1} + e_n|} \right\rangle \leq \dfrac{f_n(x_n)}{\sigma_n(|b+\xi_n|)} \quad \mbox{ in } \quad B_1.
\]
Therefore, by passing the limit as $n \to \infty$, we obtain that $u_\infty$ is a viscosity subsolution to \eqref{eq_scallim2}. We can prove analogously that $u_\infty$ is also a viscosity supersolution. This finishes the proof.
\end{proof}

\section{Differentiability of solutions}

This section presents the proof of Theorem \ref{theo_1}, namely, we show that solutions to \eqref{1} are differentiable, with appropriate estimates. Recall that ${\mathcal N}$ is a family of non-collapsing sets constructed in Section \ref{crafting_sequences}. 

\begin{lemma}[Approximation Lemma]\label{h_function}
Let $\sigma\in\mathcal{N}$ and $u\in C(B_1)$ be a normalized viscosity solution to \eqref{eq_scal}. Assume that $A\ref{Asigma}$ and $A\ref{Af}$ are in force. Given $\delta>0$, there exists $\varepsilon>0$ such that if 
\[
\|f\|_{L^{\infty}(B_1)}\leq\varepsilon,
\]
then we can find a function $h\in C^{1,\beta}(B_{1/2})$ satisfying
\[
\|u-h\|_{L^{\infty}(B_{1/2})}\leq\delta \quad \text{and} \quad \|h\|_{C^{1,\beta}(B_{1/2})}\leq C,
\]
where the constants $C$ and $\beta$ are independent of $\mathcal{N},\delta$ and $\varepsilon$.
\end{lemma}
\begin{proof}
We argue by contradiction. Suppose there exists $\delta_0>0$ and sequences $(\sigma_n)_{n\in \mathbb{N}},(\xi_n)_{n\in \mathbb{N}},(u_n)_{n\in \mathbb{N}}$ and $(f_n)_{n\in \mathbb{N}}$, satisfying 
\[
\sigma_n(0)=0, \quad \sigma_n(1)\geq 1 \quad \text{and} \quad \text{if}\quad \sigma_n(a_n)\to 0 \quad \text{then} \quad a_n\to 0,
\]
\[
\|f\|_{L^{\infty}(B_1)}\leq\frac{1}{n},
\]
and
\begin{equation}\label{eq_seq}
-\sigma_n(|Du_n+\xi_n|)\NDelta u_n=f_n \quad \text{in} \quad B_1,
\end{equation}
but
\begin{equation}\label{contra_1}
\sup_{B_{1/2}}|u_n-h|>\delta_0,
\end{equation}
for every $h\in C^{1,\beta}(B_{1/2})$. From Proposition \ref{local_holder_continuity}, we know that there exists a function $u_{\infty}\in C^{\beta}_{\text{loc}}(B_1)$ such that, up to a subsequence, $u_n\to u_{\infty}$ uniformly locally.

If the sequence $(\xi_n)$ is bounded, then by the lemma \ref{prop_stab}, the function $u_{\infty}$ solves 
\[
-\Delta u_\infty -(p-2)\left\langle D^2u_\infty \dfrac{Du_\infty + \xi_\infty}{|Du_\infty + \xi_\infty|}, \dfrac{Du_\infty + \xi_\infty}{|Du_\infty + \xi_\infty|} \right\rangle = 0 \quad \mbox{ in } \quad B_{4/5}, 
\]
for some $\xi_\infty \in \R^d$. As a consequence, \cite[Lemma 3.2]{APR17}  give to us a existence of $\beta_1>0$ such that $u_{\infty}\in C^{1,\beta_1}(B_{1/2})$. By choosing $h\equiv u_{\infty}$, for $n$ large enough we get a contradiction with \eqref{contra_1}. On the other hand, if the $(\xi_n)$ is unbounded, by using again Lemma \ref{prop_stab} we have that $u_{\infty}$ is a viscosity solution to
\[
-\Delta u_\infty -(p-2)\left\langle D^2 u_\infty e_\infty, e_\infty \right\rangle = 0 \quad \mbox{ in } \quad B_{4/5},
\]
where $|e_\infty| = 1$. Observe that this operator is linear and uniformly elliptical with constant coefficients. Therefore, from \cite[Corollary 5.7]{Caffarelli-Cabre-1995} we have that $u_{\infty}\in C^{1,\beta_0}(B_{1/2})$ for a universal $\beta_0\in (0,1)$. Once more, by choosing $h\equiv u_{\infty}$, we obtain a contradiction with \eqref{contra_1} for $n$ sufficiently large.  
\end{proof}

Next, we use the function $h \in C^{1, \beta}(B_{1/2})$ to find an affine function that is close to $u$ in the $L^\infty$-norm.   
 
\begin{proposition}\label{approximation_lemma}
Let $\sigma\in\mathcal{N}$ and $u\in C(B_1)$ be a normalized viscosity solution to \eqref{eq_scal}. Assume that $A\ref{Asigma}$-$A\ref{Af}$ hold true. There exists $\varepsilon>0$ such that if 
\[
\|f\|_{L^{\infty}(B_1)}\leq\varepsilon,
\]
then we can find an affine function of the form ${L}(x):= A + B\cdot x$, with $|A|, |B| \leq C$, satisfying
\[
\sup_{B_\rho}|u(x)-{L}(x)|\leq\mu_1\cdot\rho,
\]
where $\mu_1$ was defined in Section \ref{crafting_sequences}.
\end{proposition}
\begin{proof}
The lemma \ref{h_function} ensures the existence of a function $h\in C^{1,\beta}(B_{1/2})$ such that
\[
\|u-h\|_{L^{\infty}(B_{1/2})}\leq \delta,
\]
where $\delta>0$ is a constant to be determined later. The regularity of $h$ implies that for any $\rho \ll 1$, there exists a universal constant $C>0$ such that
\[
\|h(x)-h(0)-Dh(0)\cdot x\|_{L^{\infty}(B_{\rho})}\leq C \rho^{1+\beta}.
\]
By defining ${L}(x)\coloneqq h(0)+Dh(0)\cdot x$, we obtain
\begin{align*}
\sup_{B_{\rho}}|u(x)-{L}(x)|
&\leq
\sup_{B_{\rho}}|u(x)-h(x)|+\sup_{B_{\rho}}|h(x)-h(0)-Dh(0)\cdot x|\\
&\leq
\delta+C\rho^{1+\beta}.
\end{align*}
Now we can choose
\[
\delta\coloneqq\frac{\mu_1}{2}\cdot \rho \quad \text{and} \quad \rho\coloneqq\left(\frac{\mu_1}{2C}\right)^{\frac{1}{\beta}},
\]
so that
\[
\sup_{B_{\rho}}|u(x)-L(x)|\leq \mu_1\cdot \rho.
\]
\end{proof}

Notice that setting the value of $\delta$ as above fixes the value of $\varepsilon$ through Lemma \ref{h_function}. Next, we iterate the previous result to find a family of affine functions that approximate $u$ on small scales. That is exactly where the auxiliary functions do not solve the same class of equations as $u$ and we ``lose" the $C^{1, \alpha}$ modulus of continuity of solutions, leading to the $C^1$-regularity.

\begin{proposition}\label{Seq_affine_functions}
Let $\sigma\in\mathcal{N}$ and $u\in C(B_1)$ be a normalized viscosity solution to \eqref{1}. Assume that $A\ref{Asigma}$-$A\ref{Af}$ are in force. If 
\[
\|f\|_{L^{\infty}(B_1)}\leq\varepsilon,
\]
then, there exists a sequence of affine functions $({L}_{n})_{n\in \mathbb{N}}$, such that
\[
\sup_{B_{\rho^{n}}}|u-L_{n}|\leq\left(\prod_{i=1}^{n}\mu_i\right)\cdot\rho^{n},
\]
where $(\mu_{n})_{n\in\mathbb{N}}$ is defined in Section \ref{crafting_sequences}. Moreover, each ${L}_{n}$ is of the form
\begin{equation}\label{ln_seq}
{L}_{n}(x)\coloneqq A_n+B_n\cdot x,
\end{equation}
and for every $n \in \N$,
\[
|A_{n+1}-A_n|\leq C \left(\prod_{i=1}^{n}\mu_i\right)\cdot\rho^{n} \quad \text{and} \quad |B_{n+1}-B_{n}|\leq C\left(\prod_{i=1}^{n}\mu_i\right).
\]
\end{proposition}
\begin{proof}
We resort to an induction argument. We define the auxiliary function $v_1: B_1 \to \R$,
\[
v_1(x)\coloneqq\frac{u(\rho x)-L(\rho x)}{\mu_1\rho},
\]
where $L$ comes from Proposition \ref{approximation_lemma}. Notice that $|v_1| \leq 1$ and since $u$ is a viscosity solution to \eqref{1}, we get that $v_1$ solves
{\small
\[
-\sigma_1\left(\left|Dv_1(x)+\mu_1^{-1}B\right|\right) \left[\tr(D^2v_1(x))+(p-2)\left\langle D^2v_1(x)\frac{Dv_1(x)+\mu_1^{-1}B}{|Dv_1(x)+\mu_1^{-1}B|},\frac{Dv_1(x)+\mu_1^{-1}B}{|Dv_1(x)+\mu_1^{-1}B|}\right\rangle\right]=f_1(x) \quad \text{in } \, B_1,
\]
}
where 
\[
\sigma_1(t)\coloneqq\frac{\mu_1}{\rho}\sigma(\mu_1t) \quad \text{and} \quad f_1(x)\coloneqq f(\rho x).
\]
Recall that we choose $\rho<\mu_1$ such that $\sigma_1(1)=1$, and therefore $v_1$ satisfies the hypotheses of Proposition \ref{approximation_lemma}. Then, there exists an affine function ${L}_1(x)\coloneqq A_1+B_1\cdot x$ such that 
\[
\sup_{B_\rho}|v_1-L_1| \leq \mu_1\cdot\rho.
\]
Now, let us define a second auxiliary function $v_2: B_1 \to R$, given by
\[
v_2(x)\coloneqq\frac{v_1(\rho x)-{L}_1(\rho x)}{\mu_2\rho},
\]
for $\mu_1\leq\mu_2$, as in Section \ref{crafting_sequences}. Then, $|v_2| \leq 1$ and $v_2$ is a viscosity solution to
{\small
\[
-\sigma_2\left(\left|Dv_2(x)+\mu_2^{-1}B_1\right|\right) \left[\tr(D^2v_2(x))+(p-2)\left\langle D^2v_2(x)\frac{Dv_2(x)+\mu_2^{-1}B_1}{|Dv_2(x)+\mu_2^{-1}B_1|},\frac{Dv_2(x)+\mu_2^{-1}B_1}{|Dv_2(x)+\mu_2^{-1}B_1|}\right\rangle\right]=f_2(x) \quad \text{in } \, B_1,
\]
}
where $f_2(x)\coloneqq f_1(\rho x)=f(\rho^2x)$ and
\[
\sigma_2(t)\coloneqq\frac{\mu_2}{\rho}\sigma_1(\mu_2t)=\frac{\mu_1\mu_2}{\rho^2}\sigma(\mu_1\mu_2t).
\]
The choice of $\mu_2$ ensures that $\sigma_2(1)=1$, and once more Proposition \ref{approximation_lemma} guarantees the existence of ${L}_2(x)\coloneqq A_2+B_2\cdot x$ such that
\[
\sup_{B_\rho}|v_2-{L}_2|\leq\mu_1\cdot\rho.
\]
Continuing the pattern inductively, we define
\[
v_{n+1}(x)\coloneqq\frac{v_{n}(\rho x)-{L}_{n}(\rho x)}{\mu_{n+1}\rho},
\]
Then, $|v_{n+1}| \leq 1$ and $v_{n+1}$ is a viscosity solution to a similar problem with the right-hand side given by $f_{n+1}(x)\coloneqq f_{n}(rx)$ and the degeneracy is of the form 
\[
\sigma_{n+1}(t)\coloneqq\frac{\mu_{n+1}}{r}\sigma_{n}(\mu_{n+1}t).
\]
Again, we recall that the choice of $\mu_{n+1}$ is such that $\sigma_{n+1}(1)=1$. As before, Proposition \ref{approximation_lemma} ensures the existence of an affine function $\mathrm{L}_{n+1}(x)$ that satisfy 
\[
\sup_{B_\rho}|v_{n+1}-{L}_{n+1}|_{L^{\infty}(B_{\rho})}\leq\mu_{1}\cdot\rho.
\]
Rescaling back to the original function $u$, we obtain
\[
\sup_{B_{\rho^{n+1}}}|u(x)-{L}_{n+1}(x)| \leq \left(\prod_{i=1}^{n+1}\mu_i\right)\rho^{n+1},
\]
where the sequence $({L}_n)_{n\in\mathbb{N}}$ is given by
\begin{equation}\label{function_L}
{L}_1(x)\coloneqq L(x) \quad \text{and}\quad {L}_{n+1}(x) \coloneqq {L}(x)+\sum_{i=1}^{n} \left[\left(\prod_{j=1}^{i}\mu_{j}\right)\rho^{i}{L}_{i}(\rho^{-i}x)\right].
\end{equation}
Moreover, from \eqref{function_L} we can infer
\[
A_{n+1}= A+\sum_{i=1}^{n}\left[\left(\prod_{j=1}^{i}\mu_{j}\right)\rho^{i}a_i\right] \quad \text{and} \quad B_{n+1} = B+\sum_{i=1}^{n}\left[\left(\prod_{j=1}^{i}\mu_{j}\right)b_i\right],
\]
as a consequence, we get
\[
|A_{n+1}-A_n|=\left|a_n\left(\prod_{j=1}^{n}\mu_j\right)\rho^n\right|\leq C\left(\prod_{j=1}^{n}\mu_j\right)\rho^n
\]
and
\[
|B_{n+1}-B_{n}|=\left|b_{n}\left(\prod_{j=1}^{n}\mu_j\right)\right|\leq C\left(\prod_{j=1}^{n}\mu_j\right).
\]
which concludes the proof.
\end{proof}

We are finally ready to present the proof of Theorem \ref{theo_1}.

\begin{proof}[Proof of Theorem \ref{theo_1}]
The construction of the sequence $(\mu_n)_{n\in\mathbb{N}}$, reveals two possibilities:\\ 

\noindent{\bf Case 1:} If there exists a number $N\in\mathbb{N}$ such that
\[
\mu_{N+1}=\mu_{N+2}=\cdots=\mu_{N+n}=\cdots,
\]
the family $\mathcal{N}=\{\sigma_n\}_{n\in\mathbb{N}}$ is not degenerate. Therefore, the classical analysis for a non-degenerate operator proves the existence of $0<\alpha<\beta$ for which local $C^{1,\alpha}$-regularity of the solutions are available. \\

\noindent{\bf Case 2:} If for an infinity amount of $N's$ hold $\mu_{N}<\mu_{N+1}$, then the crafting of the sequence in this scenario implies 
\[
\frac{\prod_{j=1}^{N+1}\mu_{j}}{r^{N+1}}\sigma\left(\left(\prod_{j=1}^{N+1}\mu_{j}\right)\cdot c_{N+1}\right)=1,
\]
and as a consequence, we get
\begin{align}\label{tau_seq1}
\prod_{j=1}^{N+1}\mu_{j}&=\frac{1}{c_{N+1}}\sigma^{-1}\left(\frac{r^{N+1}}{\prod_{j=1}^{N+1}\mu_{j}}\right)\\
&\leq
\frac{1}{c_{N+1}}\sigma(\theta^{N+1}),
\end{align}
where in the last inequality, we used that
\[
\theta\coloneqq\frac{r}{\mu_1}<1 \quad \text{and} \quad \left(\mu_1^{N+1}\right)^{-1}\geq \left(\prod_{j=1}^{N+1}\mu_{j}\right)^{-1}.
\]
We define the sequence $(\tau_n)_{n\in\mathbb{N}}$ given by
\[
\tau_n\coloneqq\prod_{j=1}^{n}\mu_{j},
\]
so from the information in Section \ref{crafting_sequences} and \eqref{tau_seq1}, we have that $\tau_n\in\ell^1$ and 
\[
\sum_{n=1}^{\infty}|\tau_n|\leq\sum_{n=1}^{\infty}\sigma^{-1}(\theta^n).
\]
Therefore, the sequences $(A_n)_{n\in\mathbb{N}}$ and $(B_n)_{n\in\mathbb{N}}$ from Proposition \ref{Seq_affine_functions} are Cauchy sequences, and so there exist $A_{\infty}$ and $B_{\infty}$ such that
\[
A_n\to A_{\infty} \quad \text{and} \quad B_n\to B_{\infty}.
\]
Define the affine function $L_{\infty}(x)\coloneqq A_{\infty}+B_{\infty}\cdot x$, and note that the Proposition \ref{Seq_affine_functions} give to us 
\[
|A_{\infty}-A_n|\leq C\sum_{i=n}^{\infty}\tau_ir^{n} \quad \text{and} \quad |B_{\infty}-A_{n}|\leq C\sum_{i=n}^{\infty}\tau_i. 
\]
Now, for any $0<\rho\ll 1$ sufficiently small take $n\in\mathbb{N}$ such that $r^{n+1}\leq \rho\leq r^{n}$, hence
\begin{align*}
\|u(x)-L_{\infty}\|_{L^{\infty}(B_{\rho})}&\leq \|u(x)-L_n(x)\|_{L^{\infty}(B_{r^n})}+\|L_n(x)-L_{\infty}(x)\|_{L^{\infty}(B_{r^n})}\\
&\leq
C\tau_{n}r^n+C\left(\sum_{i=n}^{\infty}\tau_i\right)r^n\\
&\leq
\frac{C}{r}\left[\tau_n+\sum_{i=n}^{\infty}\tau_i\right]\rho\\
&\leq 
C\left(\sum_{i=n}^{\infty}\tau_i\right)\rho.
\end{align*}
To finish, let us set 
\[
\gamma(t)\coloneqq C \sum_{i=n(t)}^{\infty}\tau_i,
\]
where $n(t)$ is the biggest integer number such that $n(t)\leq \ln({t^{-1}})$. Since $\tau_n\in \ell^1$, we know from the characterization present in the assumption $A\ref{Adini}$ that $\gamma(t)$ is a modulus of continuity satisfying the Dini condition, therefore
\[
\sup_{B_{\rho}}|u(x)- A_{\infty}-B_{\infty}\cdot x|\leq \gamma(\rho)\rho.
\]
We conclude the proof taking $\rho\to 0$ to ensure that $A_{\infty}=u(0)$ and $B_{\infty}=Du(0)$.
\end{proof}

\textbf{Acknowledgements.} C. Alcantara was partially supported by the Coordenação de Aperfeiçoamento de Pessoal de Nível Superior (CAPES) – Brasil and by Stone Instituição de Pagamento S.A., under the Project StoneLab. This study was partly financed by the CAPES—Brazil - Finance Code 001. M. Santos is partially supported by the Portuguese government through FCT-Funda\c c\~ao para a Ci\^encia e a Tecnologia, I.P., under the projects UID/04561/2025 and the Scientific Employment Stimulus - Individual Call (CEEC Individual), DOI identifier https://doi.org/10.54499/2023.08772.CEECIND/CP2831/CT0003. 

\bibliographystyle{amsplain}
\bibliography{Alcantara_Santos}

\bigskip

\noindent\textsc{Claudemir Alcantara}\\
Department of Mathematics\\
Pontifical Catholic University of Rio de Janeiro (PUC-Rio), \\
22451-900, Rio de Janeiro, Brazil\\
\noindent\texttt{alcantara@mat.puc-rio.br}

\vspace{.15in}

\noindent\textsc{Makson S. Santos}\\
Center for Mathematical Studies (CEMS.UL)\\
University of Lisbon\\
1749-016 Lisboa, Portugal\\
\noindent\texttt{msasantos@ciencias.ulisboa.pt}

\end{document}